\documentclass[amssymb, 11pt]{amsart}
\usepackage{latexsym}

\newlength{\standardunitlength}
\newtheorem{prop}{Proposition}[section]

\newtheorem{lemma}[prop]{Lemma}

\newtheorem{theorem}[prop]{Theorem}

\begin{document}

\title [Representation theory] {Representation theory and central limit theorems for traces of commutators for compact Lie groups}

\author{Jason Fulman}
\address{Department of Mathematics, University of Southern California, Los Angeles, CA 90089-2532, USA }
\email{fulman@usc.edu}



\date{February 12, 2025}

\begin{abstract} There has been some work in the literature on limit theorems for the trace of
commutators for compact Lie groups. We revisit this from the perspective of combinatorial representation theory.
\end{abstract}

\maketitle

\section{Introduction}

Classic work of Diaconis and Shahshahani \cite{DS} proved central limit theorems for the trace of Haar distributed
elements of the compact Lie groups $U(n,C), Sp(2n,C), O(n,R)$. Analogous results for traces of commutators have
been obtained by Palheta, Barbosa, and Novaes \cite{PBN} and others \cite{Ra}, \cite{MSS}, \cite{MiP}, \cite{Re} using
tools such as the Weingarten calculus and free probability theory. We revisit these results on commutators from
the viewpoint of combinatorial representation theory. Most of our results are known, but our tools are different and
in our opinion worth recording.

We should mention that our technique makes heavy use of representation theory and does not apply to general
words. For a guide to the literature on more general words, see the breakthrough papers \cite{MP1}, \cite{MP2},
\cite{Vo} and the references therein. We also note that the results in this paper can be regarded as a ``compact Lie group
analog'' of some of the results in our earlier paper \cite{Fu}. Whereas \cite{Fu} studied fixed points of non-uniform permutations,
the current paper studies traces of non-Haar distributions on compact Lie groups.

This paper is organized as follows. Section \ref{prel} gives needed preliminaries on commutators and characters
of compact Lie groups. Everything in this section is known, but we have been unable to find in the literature proofs of
two of the crucial tools (Lemma \ref{second} and Theorem \ref{com}), so we include proofs. Section
\ref{symplec} proves a central limit theorem for the trace of products of commutators for the symplectic groups.
Section \ref{orthogr} gives an analog for orthogonal groups and Section \ref{unitgr} gives an analog for
unitary groups. We suspect, but are not certain, that the central limit theorems in these sections are known, so to be
safe we do not claim that the results are original and only claim that the approach (using results of Sundaram
\cite{Su} and Stembridge \cite{Stem} in the algebraic combinatorics literature) is new.

To close the introduction we mention some notation to be used throughout this paper. We let $G$ be a compact Lie group
and all integrals over $G$ are with respect to Haar measure. We let $\chi$ denote an irreducible character of $G$ and let
$d_{\chi}$ denote its dimension. We let $[x,y]$ denote the commutator $x^{-1}y^{-1}xy$. For an integer partition $\lambda$,
we let $l(\lambda)$ denote its number of parts.

\section{Preliminaries} \label{prel}

The results in this section will be useful for studying the products of commutators. These results are all known, but we have been unable to find proofs of Lemma \ref{second} and  Theorem \ref{com} in the literature. Since we use these results crucially, we include proofs.

Lemma \ref{fir} is proved on page 392 of \cite{He}. 

\begin{lemma} \label{fir} Let $G$ be a compact Lie group and $\chi$ an irreducible character of $G$. Then for all
$\alpha,g \in G$
\[ \int_G \chi(y^{-1} \alpha y g) dy = \frac{\chi(\alpha) \chi(g)}{d_{\chi}}.\]
\end{lemma}

We encountered Lemma \ref{second} in Kefeng Liu's paper \cite{Ke}. Kefeng Liu (personal communication) has provided the following
proof due to his student Tiancheng Xia.

\begin{lemma} \label{second} Let $G$ be a compact Lie group and $\chi$ an irreducible character of $G$. Then for all $w \in G$,
\[ \int_G \chi(wy) \chi(y^{-1}) dy = \frac{\chi(w)}{d_{\chi}}.\]
\end{lemma}

\begin{proof} Let $<,>$ be a $G$-invariant inner product on the space $V$ corresponding to $\chi$, and let $v_1,\cdots,v_{d_{\chi}}$
be an orthonormal basis with respect to $<,>$. 

Then
\begin{eqnarray*}
\int_G \chi(wy) \chi(y^{-1}) dy & = & \int_G \sum_i <wyv_i,v_i> \sum_j <y^{-1}v_j,v_j> dy \\
& = & \int_G \sum_i <yv_i,w^{-1}v_i> \sum_j <y^{-1}v_j,v_j> dy \\
& = & \sum_{i,j} \int_G <yv_i,w^{-1}v_i> <y^{-1}v_j,v_j> dy.
\end{eqnarray*}

Part 2 of Theorem 4.5 of \cite{BD} gives the orthogonality relation
\[ \int_G <y^{-1}v,\alpha> <y \beta,w> dy = \frac{1}{d_{\chi}} <\beta,\alpha> <v,w> \] for
any $\alpha,\beta,v,w$.
Thus
\begin{eqnarray*}
\int_G \chi(wy) \chi(y^{-1}) dy & = & \frac{1}{d_{\chi}} \sum_{i,j} <v_i,v_j> <v_j,w^{-1}v_i> \\
& = & \frac{1}{d_{\chi}} \sum_i <v_i,w^{-1}v_i> \\
& = & \frac{1}{d_{\chi}} \sum_i <wv_i,v_i> = \frac{\chi(w)}{d_{\chi}}.
\end{eqnarray*}
\end{proof}

As a consequence of these lemmas, we obtain the following theorem stated in \cite{Ke}.

\begin{theorem} \label{com} Let $G$ be a compact Lie group and $\chi$ an irreducible character of $G$. Then
\[ \int_{G^{2k}} \chi([y_1,z_1] \cdots [y_k,z_k]) dy_1 dz_1 \cdots dy_k dz_k = \frac{1}{d_{\chi}^{2k-1}}.\] 
\end{theorem}

\begin{proof} We proceed by induction on $k$. For the case $k=1$, 
\[ \int_{G^2} \chi(y^{-1}z^{-1}yz) dy dz = \frac{1}{d_{\chi}} \int_G \chi(z^{-1}) \chi(z) dz = \frac{1}{d_{\chi}},\]
where the equalities used Lemmas \ref{fir} and \ref{second}. 

Now we compute

\begin{eqnarray*}
& & \int_{G^{2k}} \chi([y_1,z_1] \cdots [y_k,z_k]) dy_1 dz_1 \cdots dy_k dz_k \\
& = & \frac{1}{d_{\chi}} \int_{G^{2k-1}} \chi(z_1^{-1}) \chi(z_1 [y_2,z_2] \cdots [y_k,z_k]) dz_1 dy_2dz_2 \cdots dy_k dz_k \\
& = &  \frac{1}{d_{\chi}} \int_{G^{2k-1}} \chi(z_1^{-1}) \chi([y_2,z_2] \cdots [y_k,z_k] z_1) dz_1 dy_2dz_2 \cdots dy_k dz_k \\
& = & \frac{1}{d_{\chi}^2} \int_{G^{2k-2}} \chi([y_2,z_2] \cdots [y_k,z_k]) dy_2 dz_2 \cdots dy_k dz_k \\
& = & \frac{1}{d_{\chi}^2} \frac{1}{d_{\chi}^{2(k-1)-1}} = \frac{1}{d_{\chi}^{2k-1}}.
\end{eqnarray*} The first equality used Lemma \ref{fir}, the third equality used Lemma \ref{second}, and the fourth equality used
the induction hypothesis.
\end{proof}

{\it Remark:} Theorem \ref{com} has an exact analog for finite groups. Indeed, for $G$ a finite group with irreducible
character $\chi$,
\[ \frac{1}{|G|^{2k}} \sum_{y_1,\cdots,y_k \atop z_1,\cdots,z_k} \chi([y_1,z_1] \cdots [y_k,z_k]) = \frac{1}{d_{\chi}^{2k-1}}.\]
This can be deduced from Theorem A.1.10 in Zagier's appendix to the book \cite{LZ}. This means that
Theorem 3.3 in \cite{Fu}, which proves a Poisson(1) limit theorem for the number of fixed points of a random commutator
in the symmetric group, extends to a Poisson limit for the product of commutators.

\section{Symplectic groups} \label{symplec}

Throughout this section $G=Sp(2n,C)$. The following definition is from Sundaram \cite{Su}.

{\bf Definition:} For $\lambda$ a partition of $n$, an up-down tableau of shape $\lambda$ and length $r$ is a sequence of shapes
$\emptyset=\lambda^0, \lambda^1,\cdots,\lambda^r=\lambda$ such that any two consecutive shapes differ by exactly one box.
We let $f_r^{\lambda}$ be the number of up-down tableaux of shape $\lambda$ and length $r$. We also define $f_r^{\lambda}(n)$
to be the number of up-down tableaux of shape $\lambda$ and length $r$ such that every shape in the sequence has length
at most $n$.

The following theorem is from Sundaram \cite{Su}.

\begin{theorem} \label{sec} For all integers $n \geq 1$ and $r \geq 0$,
\[ (x_1+x_1^{-1}+\cdots+x_n+x_n^{-1})^r = \sum_{\lambda \atop l(\lambda) \leq n} f_r^{\lambda}(n)
\chi_{sp}^{\lambda} (x_1^{\pm 1},\cdots,x_n^{\pm 1}), \] where $\chi_{sp}^{\lambda}$ is the irreducible character of
$Sp(2n,C)$ corresponding to $\lambda$ and $l(\lambda)$ is the length of $\lambda$.
\end{theorem}

Theorems \ref{com} and \ref{sec} immediately yield the following result. In its statement, $Tr$ denotes trace.

\begin{theorem} \label{firthm} Let $G=Sp(2n,C)$ and let $k$ and $r$ be fixed. Let $x_1,\cdots,x_k$ and $y_1,\cdots, y_k$ be
chosen independently from the Haar measure of $G$. Then

\begin{enumerate}
\item \[ \int_{G^{2k}} [Tr([x_1,y_1] \cdots [x_k,y_k])]^r dx_1 \cdots dx_n dy_1 \cdots dy_n =
\sum_{\lambda \atop l(\lambda) \leq n} \frac{f_r^{\lambda}(n)}{d_{\lambda}^{2k-1}}.\]

\item For sufficiently large $n$,  \[  \int_{G^{2k}} [Tr([x_1,y_1] \cdots [x_k,y_k])]^r dx_1 \cdots dx_n dy_1 \cdots dy_n  = \sum_{\lambda} \frac{f_r^{\lambda}}{d_{\lambda}^{2k-1}}.\]
\end{enumerate} 

\end{theorem}

As a consequence of the previous theorem, we obtain a central limit theorem for the trace of the product of $k$ commutators of random elements
of $Sp(2n,C)$. This is stated in \cite{PBN} for $k=1$ but their paper only provides details for the unitary and orthogonal groups (and their
method is different). However we suspect that Theorem \ref{mainsym} may be known.

\begin{theorem} \label{mainsym} As $n \rightarrow \infty$, the trace of the product $[x_1,y_1] \cdots [x_k,y_k]$ where $x_1,\cdots,x_k$ and $y_1,\cdots,y_k$ are
chosen independently from Haar measure of $Sp(2n,C)$ tends to a standard normal. \end{theorem}

\begin{proof} We use the method of moments \cite{Di}. By part 2 of Theorem \ref{firthm}, it is enough to show that for $r$
fixed and $n \rightarrow \infty$, the quantity $\sum_{\lambda} \frac{f_r^{\lambda}}{d_{\lambda}^{2k-1}}$ tends to $0$ if $r$
is odd and to $(r-1)(r-3) \cdots (3)(1)$ if $r$ is even.

For any $n$, the number of non-zero terms in $\sum_{\lambda} \frac{f_r^{\lambda}}{d_{\lambda}^{2k-1}}$ is upper bounded in terms
of $r$. 

Consider the term corresponding to $\lambda=\emptyset$. From Theorem 8.7 of \cite{St} this term is equal to $0$ if $r$
is odd and to $(r-1)(r-3) \cdots (3)(1)$ if $r$ is even. So to complete the proof, it suffices to show that for all $\lambda \neq
\emptyset$, \begin{equation} \label{toshow} \frac{f_r^{\lambda}}{d_{\lambda}^{2k-1}} \leq A_r/n \end{equation} for a constant
$A_r$ depending on $r$ but not on $n$. Clearly $f_r^{\lambda} \leq (2r)^r$, since for each shape in the sequence
$\emptyset=\lambda^0, \lambda^1,\cdots,\lambda^r=\lambda$  there are two choices (to add or remove a box from the
previous shape) and at most $r$ choices for the row to which the box is added or removed.

So to prove \eqref{toshow} it is enough to show that if $f_r^{\lambda} \neq 0$ and $\lambda \neq \emptyset$, then $d_{\lambda} \geq B_r n$ for a
non-zero constant $B_r$ depending on $r$ but not on $n$. Since $f_r^{\lambda} \neq 0$ implies that the size of $\lambda$ is at
most $r$, this follows immediately from a formula of El-Samra and King \cite{EK} (Theorem 4.5 of \cite{Su}) which states that for $n \geq l(\lambda)$
\[ d_{\lambda} = \prod_{x \in \lambda} \frac{2n-t(x)}{h(x)}. \] Here the product is over the boxes of the diagram of $\lambda$
and $t(x)$ and $h(x)$ are combinatorial quantities upper and lower bounded in terms of $r$. This completes the proof.
\end{proof}

{\it Remark:} The proof of Theorem \ref{mainsym} shows that the expected value of the trace of  $[x_1,y_1] \cdots [x_k,y_k]$
is $O(1/n^{2k-1})$. A similar remark applies to the orthogonal and unitary cases and is consistent with results of Magee and Puder
\cite{MP1},\cite{MP2}.

\section{Orthogonal groups} \label{orthogr}

In this section we study the trace of the product of commutators for the orthogonal groups. We choose to work with the special orthogonal
groups $SO(n,C)$ since in this case the crystal clear exposition of Sundaram \cite{Su} is perfectly suited to our needs. Our results can also
be modified to treat the groups $O(n,C)$, using results of Proctor \cite{Pr}.

To begin we treat the even dimensional groups $SO(2n,C)$. Theorem \ref{Sunorth} is Theorem 3.14 of Sundaram \cite{Su}.

\begin{theorem} \label{Sunorth} For all integers $r \geq 0$ and $n>r$,
\[ (x_1+x_1^{-1} + \cdots + x_n+x_n^{-1})^r = \sum_{\lambda} f_r^{\lambda} \chi_{so}^{\lambda}(x_1^{\pm 1},\cdots,x_n^{\pm 1}),\]
where $f_r^{\lambda}$ is as in Section \ref{symplec} and $\chi_{so}^{\lambda}$ is the irreducible character of $SO(2n,C)$ corresponding to $\lambda$.
\end{theorem}

{\it Remark:} As explained in \cite{Su}, since $n>r$, the characters $\chi_{so}^{\lambda}$ with $f_r^{\lambda} \neq 0$ are indeed
irreducible.

This leads to the following central limit theorem.

\begin{theorem} \label{CLTSOeven} Fix $k \geq 1$. As $n \rightarrow \infty$, the trace of the product $[x_1,y_1] \cdots [x_k,y_k]$ where
$x_1,\cdots,x_k$ and $y_1,\cdots,y_k$ are chosen independently from the Haar measure of $SO(2n,C)$ tends to a standard normal.
\end{theorem}

\begin{proof} We argue exactly as for the symplectic groups. Since $f_r^{\lambda}$ is the same as in the symplectic case, all that is
left to show is that $d_{\lambda} \geq B_r n$ for a non-zero constant $B_r$ depending on $r$ but not on $n$. This follows immediately
from a formula of El-Samra and King \cite{EK} (Theorem 4.5 of \cite{Su}) which states that for $n \geq l(\lambda)$,
\begin{equation} \label{starone} d_{\lambda} = \prod_{x \in \lambda} \frac{2n + s(x)}{h(x)} \end{equation} for certain quantities
$s(x)$ and $h(x)$ which are upper and lower bounded in terms of $r$. \end{proof}

The identical result holds for odd dimensional orthogonal groups.

\begin{theorem} Fix $k \geq 1$. As $n \rightarrow \infty$, the trace of the product $[x_1,y_1] \cdots [x_k,y_k]$ where
$x_1,\cdots,x_k$ and $y_1,\cdots,y_k$ are chosen independently from the Haar measure of $SO(2n+1,C)$ tends to a standard normal.
\end{theorem}

\begin{proof} From Theorem 3.15 of \cite{Su} (proved in \cite{Su2}), for all $r \geq 0$ and large enough $n$,
\[ (x_1+x_1^{-1} + \cdots + x_n+x_n^{-1} + 1)^r = \sum_{\lambda} f_r^{\lambda} \chi_{so}^{\lambda}(x_1^{\pm 1},\cdots,x_n^{\pm 1},1),\]
where  $f_r^{\lambda}$ is as in Section \ref{symplec} and $\chi_{so}^{\lambda}$ is the irreducible character of $SO(2n+1,C)$ corresponding to $\lambda$. So we can argue as in the even dimensional case, and we are done since for $n \geq l(\lambda)$ the dimension formula \eqref{starone} holds
with $2n$ replaced by $2n+1$.
\end{proof}

\section{Unitary groups} \label{unitgr}

This section treats the unitary groups $U(n,C)$. Unlike the symplectic and orthogonal cases, here the trace of the product of
commutators converges to a standard complex normal. (A random variable $Z$ is said to be a standard complex normal if
$Z=X+iY$ where $X$ and $Y$ are independent real normals with mean $0$ and variance $1/2$).

We will make extensive use of Stembridge's lovely paper \cite{Stem} and begin with some definitions from that paper.

{\bf Definition:} We define a staircase $\gamma$ of height $n$ to be an integer sequence
$\gamma_1 \geq \cdots \geq \gamma_n$. We define $|\gamma|= \gamma_1+\cdots+\gamma_n$.

It is known (Chapter 38 of \cite{Bu}) that staircases of height $n$ index the irreducible representations of $U(n,C)$. We denote
the corresponding character by $\chi^{\gamma}_{u}(x_1,\cdots,x_n)$.

The following definition is also from \cite{Stem}.

{\bf Definition:} An up-down staircase tableau $T$ of length $j$ and shape $\gamma$ is a sequence $\emptyset=\gamma^0, \gamma^1,
\cdots,\gamma^j=\gamma$ of staircases of height $n$ in which either $\gamma^i \supset \gamma^{i-1}$ ,$|\gamma^i| - |\gamma^{i-1}|=1$
or $\gamma^i \subset \gamma^{i-1}$ , $|\gamma^i| - |\gamma^{i-1}|=-1$ for all $1 \leq i \leq j$. The tableau $T$ is said to be of
type $\epsilon=(\epsilon_1,\cdots,\epsilon_j)$ where $\epsilon_i=|\gamma^i|-|\gamma^{i-1}|$. We let $c_j^{\gamma}(\epsilon)$ 
be the number of of up-down staircase tableaux of length $j$, shape $\gamma$, and type $\epsilon$.

As an example, the sequence
\[ \emptyset, (0,0,-1), (1,0,-1), (1,0,0), (1,1,0), (1,1,-1) \] defines an up-down staircase tableau of length 5, type $(-1,1,1,1,-1)$ and
shape $(1,1,-1)$.

Theorem \ref{complexdec} is equation 9 of \cite{Stem}.

\begin{theorem} \label{complexdec} For all integers $n \geq 1$ and $r,s >0$,
\[ (x_1+\cdots+x_n)^r (x_1^{-1}+\cdots+x_n^{-1})^s = \sum_{\gamma} c^{\gamma}_{r+s}(\epsilon) \chi^{\gamma}_u(x_1,\cdots,x_n),\]
where the sum is over all staircases $\gamma$ of height $n$ (so over irreducible representations of $U(n,C)$), and $\epsilon$ denotes any fixed
type vector with $+1$ occurring $r$ times and $-1$ occurring $s$ times. \end{theorem}

The following theorem is an immediate consequence of Theorem \ref{com}, Theorem \ref{complexdec}, and the fact
that all eigenvalues $x_i$  of a unitary matrix lie on the unit circle (so that $x_i^{-1} = \overline{x_i}$). Here, and throughout the remainder
 of this section, $\overline{z}$ denotes the complex conjugate of a complex number $z$.

\begin{theorem} \label{ana} Let $G=U(n,C)$ and let $k,r,s$ be fixed. Let $x_1,\cdots,x_k$ and $y_1,\cdots,y_k$ be chosen
independently from the Haar measure of $G$. Then 
\begin{eqnarray*}
& & \int_{G^{2k}} [Tr([x_1,y_1] \cdots [x_k,y_k])]^r \overline{[Tr([x_1,y_1] \cdots [x_k,y_k])] }^s dx_1 dy_1 \cdots dx_k dy_k\\
& = & \sum_{\gamma} \frac{c_{r+s}^{\gamma}(\epsilon)}{d_{\gamma}^{2k-1}}.
\end{eqnarray*}
\end{theorem}

Now we come to the main result of this section.

\begin{theorem} Fix $k \geq 1$. As $n \rightarrow \infty$, the trace of the product $[x_1,y_1] \cdots [x_k,y_k]$ where
$x_1,\cdots,x_k$ and $y_1,\cdots,y_k$ are chosen independently from Haar measure of $U(n,C)$ tends to a standard complex normal.

\begin{proof} It is known (Lemma 1 of \cite{DS}) that if $X$ is a standard complex normal, then for all natural numbers $r,s$,
$E[X^r \overline{X}^s] = \delta_{r,s} r!$. So from Theorem \ref{ana}, it is enough to show that as $n \rightarrow \infty$
\begin{equation} \label{star} \sum_{\gamma}  \frac{c_{r+s}^{\gamma}(\epsilon)}{d_{\gamma}^{2k-1}} \end{equation}
tends to $r!$ if $r=s$ and to $0$ otherwise.

First consider the case $r=s$. Consider the term $\gamma=\emptyset$. Then $d_{\emptyset}=1$ and by Proposition 4.8
of \cite{Stem}, $c_{r+s}^{\emptyset}(\epsilon)=r!$ for large enough $n$. Consider the other terms in \eqref{star}. The
number of non-zero terms is upper bounded in terms of $r$, and each $c^{\gamma}_{r+s}(\epsilon)$ is also upper bounded
in terms of $r$. So it suffices to show that if $c^{\gamma}_{r+s}(\epsilon) \neq 0$ and $\gamma \neq \emptyset$, then
$d_{\gamma} \geq B_r n$ for a non-zero constant $B_r$ depending on $r$ but not on $n$. This follows immediately
from the dimension formula for irreducible representations of $GL(n,C)$ (see Theorem 4.4 of \cite {Su}) together
with the fact (Chapter 38 of \cite{Bu}) that the irreducible characters of $U(n,C)$ are given by multiplying the irreducible
characters of $GL(n,C)$ by a power of $x_1 \cdots x_n$, which doesn't change their dimension.

The second case is that $r \neq s$. Note that $c_{r+s}^{\emptyset}(\epsilon)=0$. As in the $r=s$ case, the number of non-zero
terms in \eqref{star} is upper bounded in terms of $r$ and $s$, and each $c_{r+s}^{\gamma}(\epsilon)$ is upper bounded in
terms of $r$ and $s$. So it suffices to show that if $n$ if large enough and $c_{r+s}^{\gamma}(\epsilon) \neq 0$,
then $d_{\gamma} \geq B_{r,s} n$
for a non-zero constant $B_{r,s}$ depending on $r$ and $s$ but not on $n$, and this follows as in the $r=s$ case.
\end{proof}

\end{theorem}

\section{Acknowledgements} This research was supported by Simons grant 917224. We thank Kefeng Liu, Sheila Sundaram, and Tiancheng Xia
for their help.


\begin{thebibliography}{A}

\bibitem{BD} Br\"{o}cker, T. and tom Dieck, T., Representations of compact Lie groups, Springer, 1985.

\bibitem{Bu} Bump, D., Lie groups, Springer, 2004.

\bibitem{Di} Diaconis, P., Applications of the method of moments in probability and statistics, {\it Proc. Symp. Applied Math.}
{\bf 37} (1987), 125-142.

\bibitem{DS} Diaconis, P. and Shahshahani, M., On the eigenvalues of random matrices, {\it J. Applied Probab.} {\bf 31}
(1994), 49-62.

\bibitem {EK} El-Samra, N. and King, R.C., Dimensions of irreducible representations of the classical Lie groups
{\it J. Phys. A:Math Gen.} {\bf 12} (1979), 2317-2328.

\bibitem{Fu} Fulman, J., Fixed points of non-uniform permutations and representation theory of the symmetric group,
arxiv:2406.12139 (2024).

\bibitem{He} Helgason, S., Groups and geometric analysis, American Math Society, Providence RI, 2000.

\bibitem{LZ} Lando, S. and Zvonkin, A., Graphs on surfaces and their applications, Springer 2004.

\bibitem{Ke} Liu, K., Heat kernels, symplectic geometry, moduli spaces and finite groups, in {\it Surveys in differential
geometry} {\bf 5} (1999), 527-542.

\bibitem{MP1} Magee, M. and Puder, D., Matrix group integrals, surfaces, and mapping class groups I: U(n), {\it
Invent. Math.} {\bf 218} (2019), 341-411.

\bibitem{MP2} Magee, M., and Puder, D., Matrix group integrals, surfaces, and mapping class groups II: O(n) and Sp(n),
{\it Math. Annalen} {\bf 388} (2024), 1437-1494.

\bibitem{MiP} Mingo, J. and Popa, M., Real second order freeness and Haar orthogonal matrices, {\it J. Math. Phys.}
{\bf 54} 051701 (2013).

\bibitem{MSS} Mingo, J., Sniady, P., and Speicher, R., Second order freeness and fluctuations of random matrices II:
unitary matrices, {\it Adv. Math.} {\bf 209} (2007), 212-240.

\bibitem{PBN} Palheta, P., Barbosa, M., and Novaes, M., Commutators of random matrices from the unitary and
orthogonal groups, {\it J. Math. Phys.} {\bf 63} 113502 (2022).

\bibitem{Pr} Proctor, R., A Schensted algorithm which models tensor representations of the orthogonal groups,
{\it Canad. J. Math.} {\bf 42} (1990), 28-49.

\bibitem{Ra} Radulescu, F., Combinatorial aspects of Conne's embedding conjecture and asymptotic distribution
of traces of products of unitaries, {\it Operator theory} {\bf 20} (2006), 197-205.

\bibitem{Re} Redelmeier, C., Quaternionic second order freeness and fluctuations of large symplectically
invariant random matrices, {\it Random Matrices Theory Appl.} {\bf 10} 2150017 (2021).

\bibitem{St} Stanley, R., Algebraic combinatorics: walks, trees, tableaux, and more, Springer 2013.

\bibitem{Stem} Stembridge, J., Rational tableaux and the tensor algebra of $gl_n$, {\it J. Combin. Theory A}
{\bf 46} (1987), 79-120.

\bibitem{Su} Sundaram, S., Tableaux in the representation theory of the classical Lie groups, in {\it Invariant theory
and tableaux}, IMA Vol. Math. Appl. 19, Springer (1990), 191-255.

\bibitem{Su2} Sundaram, S., Orthogonal tableaux and an insertion algorithm for $SO(2n+1)$, {\it J.
Combin. Theory A} {\bf 53} (1990), 239-256.

\bibitem{Vo} Voiculescu, D., Limit laws for random matrices and free products, {\it Invent. Math.} {\bf 104} (1991),
201-220.

\end{thebibliography}
\end{document}